\documentclass[12pt,english]{amsart}

 \usepackage{amsfonts,mathrsfs}
 \usepackage{amsthm}
 \usepackage{amssymb}

 \usepackage[paperwidth=210mm,paperheight=297mm,inner=3cm,outer=3cm,top=3cm,bottom=3cm,marginparwidth=1.8cm]{geometry}

\usepackage[T1]{fontenc}
\usepackage{amssymb}

\usepackage[leqno]{amsmath}
\usepackage[pdftex]{graphicx}    
\usepackage{amsthm}
\usepackage{amsrefs}
\usepackage{babel}
\usepackage{caption}
\usepackage{etex}
\usepackage{fancyvrb}
\usepackage[colorlinks,linkcolor=black,citecolor=black,urlcolor=black,hypertexnames=true]{hyperref}
\usepackage[retainorgcmds]{IEEEtrantools}
\usepackage{paralist}
\usepackage{tikz}
\usetikzlibrary{shapes,patterns,calc,snakes}

\usepackage{subcaption}
\usepackage{graphicx}

\newcommand\mbb{\mathbb}

\newcommand\mcal{\mathcal}

\newcommand\ol{\overline}

\newcommand\wt{\widetilde}

\newcommand\sS{\mcal{S}}

\newcommand\sV{\mcal{V}}

\newcommand\C{\mbb{C}}

\renewcommand\P{\mbb{P}}
\newcommand\Q{\mbb{Q}}
\newcommand\R{\mbb{R}}

\renewcommand\epsilon{\varepsilon}
\renewcommand\ge{\geqslant}

\renewcommand\le{\leqslant}
\renewcommand\mod{\mathop{\rm mod}}
\renewcommand\phi{\varphi}

\renewcommand\theta{\vartheta}
\renewcommand\div{\mathop{\rm div}}

\newcommand\suchthat{\ \colon\ }

\numberwithin{equation}{section}

\newtheoremstyle{smallthm}
{1em}
{1em}
{\small}
{0pt}
{\small\bfseries}
{}
{0pt}
{\thmname{#1}\thmnumber{ #2}\thmnote{ #3}.\enspace}

\theoremstyle{remark}

\theoremstyle{plain}
\newtheorem{Thm}[equation]{Theorem}

\newtheorem{Cor}[equation]{Corollary}
\newtheorem{Lemma}[equation]{Lemma}

\newtheorem*{Thm*}{Theorem}
\newtheorem*{Prop*}{Proposition}
\newtheorem*{Cor*}{Corollary}
\newtheorem*{Lemma*}{Lemma}
\newtheorem*{Sublemma*}{Sublemma}
\newtheorem*{Conjecture*}{Conjecture}

\theoremstyle{definition}

\newtheorem{Def}[equation]{Definition}

\newtheorem{Example}[equation]{Example}

\newtheorem{Question}[equation]{Question}
\newtheorem{Remark}[equation]{Remark}

\newtheorem*{Def*}{Definition}
\newtheorem*{Defs*}{Definitions}
\newtheorem*{Example*}{Example}
\newtheorem*{Examples*}{Examples}
\newtheorem*{LemmaDef*}{Lemma and Definition}
\newtheorem*{Notation*}{Notation}
\newtheorem*{Problem*}{Problem}
\newtheorem*{Question*}{Question}
\newtheorem*{Remark*}{Remark}
\newtheorem*{Remarks*}{Remarks}
\newtheorem*{Warning*}{Warning}

\theoremstyle{smallthm}

\newtheorem*{Exercise*}{Exercise}
\numberwithin{Exercise}{section}

\begin{document}

\title[]{Spectrahedral representations of plane hyperbolic curves}

\begin{abstract}
 We describe a new method for constructing a spectrahedral representation of the hyperbolicity region of a  hyperbolic curve in the real projective plane. As a consequence, we show that if the curve is smooth and defined over the rational numbers, then there is a spectrahedral representation with rational matrices. This generalizes a classical construction for determinantal representations of plane curves due to Dixon and relies on the special properties of real hyperbolic curves that interlace the given curve.
\end{abstract}

\author{Mario Kummer}
\author{Simone Naldi}
\author{Daniel Plaumann}

\maketitle

\section*{Introduction}

Determinantal representations of plane curves are a classical topic in algebraic geometry. Given a form $f$ (i.e.~a homogeneous polynomial) of degree $d$ in three variables with complex coefficients and a general form $g$ of degree $d-1$, there exists a $d\times d$ linear matrix $M = xA+yB+zC$ such that $f$ is the determinant of $M$ and $g$ a principal minor of size $d-1$ (see for example \cite[Ch.~4]{dolgachev2012classical}). The matrix $M$ can be chosen symmetric if $g$ is a \emph{contact curve}, which means that all intersection points between the curves defined by $f$ and $g$ have even multiplicity. The construction of $M$ from $f$ and $g$ is due to Dixon \cite{dixon1902note} (following Hesse's much earlier study of the case $d=4$). We refer to this construction as the \emph{Dixon process}. 

For real curves, the most interesting case for us is that of \emph{hyperbolic curves}. The smooth hyperbolic curves are precisely the curves whose real points contain a set of $\lfloor \frac d2\rfloor$ nested ovals in the real projective plane (plus a pseudo-line if $d$ is odd). A form $f \in \R[x,y,z]$ is hyperbolic if and only if it possesses a real symmetric determinantal representation $f=\det(M)$ such that $M(e)=e_1A+e_2B+e_3C$ is (positive or negative) definite for some point $e\in\P^2(\R)$. This is the \emph{Helton--Vinnikov theorem}, which confirmed a conjecture by Peter Lax \cite{HV07}.

The Helton--Vinnikov theorem received a lot of attention in the context of semidefinite programming, which was also part of the original motivation: The set of points $a \in \R^3$ for which the matrix $M(a)$ is positive semidefinite is a \emph{spectrahedron}:
$$
\mathcal{S}(M) = \{a \in \R^3 : M(a) \succeq 0\}.
$$
It coincides with the \emph{hyperbolicity cone} $C(f,e)$ of $f=\det M$ in direction $e$, that is the closure of the connected component of $\{a \in \R^3 : f(a) \neq 0\}$ containing $e$. This is a convex cone in $\R^3$, whose image in $\P^2$ is the region enclosed by the convex innermost oval of the curve (see Figure \ref{fig:firstfigure}). A triple of real symmetric matrices $A,B,C$ is a {\it spectrahedral representation} of $C(f,e)$ if $M=xA+yB+zC$ satisfies
$$
C(f,e) = \mathcal{S}(M).
$$

\begin{figure}[!ht]
\centering
\includegraphics[width=8cm]{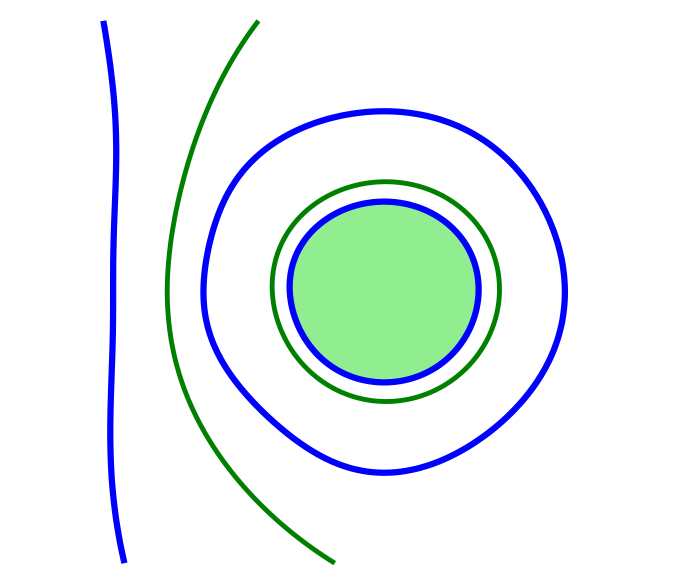}
\caption{A quintic hyperbolic curve (blue), a quartic interlacer (green), and the hyperbolicity region (green region)}
\label{fig:firstfigure}
\end{figure}

It has been pointed out by several authors \cites{Vppf,PV13} that the proof of the Helton--Vinnikov theorem becomes much simpler if one requires the matrix $M$ to be only hermitian, rather than real symmetric. In that case, $M$ can be constructed via the Dixon process starting from any \emph{interlacer} of $f$: That is, any hyperbolic form $g$ of degree $d-1$ whose ovals are nested between those of the curve defined by $f$ (see Figure \ref{fig:firstfigure}).
One downside of this apparent simplification is that the corresponding determinantal representation $f=\det(M)$ with principal minor $g$ is harder to construct explicitly, since one has to find the intersection points of $f$ and $g$, while this can be avoided if $g$ is a contact curve. We refer to \cite{Vppf} for a survey of these results.

\bigskip
In this paper, we study a modification of the Dixon process, which can be described as follows: Given a form $f$ of degree $d$, hyperbolic with respect to $e$, and an interlacer $g$ of degree $d-1$, we construct a real symmetric matrix pencil $M$ with the following properties:
\begin{itemize}
\item The determinant $\det(M)$ is divisible by $f$.
\item The principal minor $\det(M_{11})$ is divisible by $g$.
\item The extra factors $\det(M)/f$ and $\det(M_{11})/g$ are products of linear forms.
\item The spectrahedron defined by $M$ coincides with $C(f,e)$.
\end{itemize}

The extra factor in our spectrahedral representation of $C(f,e)$ is an arrangement of real lines, as in Figure \ref{fig:extralines}. Informally speaking, these additional lines correct the failure of $g$ to be a contact curve by passing through the intersection points of $g$ with $f$ that are not of even multiplicitity.

\begin{figure}[!ht]
\centering
\includegraphics[width=10cm]{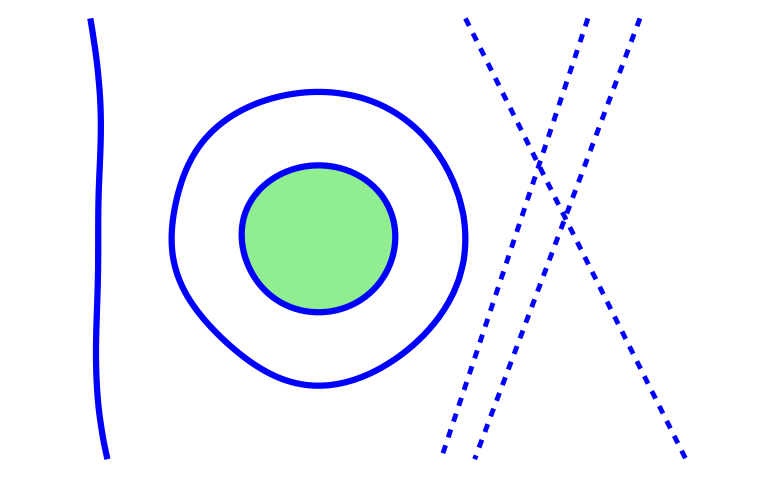}
\caption{The extra factor (dashed blue lines) giving the spectrahedral representation of the hyperbolicity region}
\label{fig:extralines}
\end{figure}

The precise statement is Theorem \ref{Thm:GenDixon}. The size of $M$ is at most quadratic in $d$. Thus while $M$ may not be the smallest or simplest determinantal representation of (some multiple of) $f$, it is easier to construct and may better reflect properties of the hyperbolicity region $C(f,e)$: as a corollary, we show that if $f$ has coefficients in $\Q$, then $C(f,e)$ can be represented by a linear matrix inequality with coefficients in $\Q$ (Cor.~\ref{thm:rationalspec}).  We may also view Theorem \ref{Thm:GenDixon} in the context of the \textit{Generalized Lax Conjecture}, which states that every hyperbolicity region (in any dimension) is spectrahedral. While various stronger forms of this conjecture have been disproved, it remains open as stated. One obstacle for constructing symmetric determinantal representations in higher dimensions is the nonexistence of contact interlacers for general hyperbolic hypersurfaces. Since our generalized Dixon process does not require the interlacer to be contact, it is possible that a spectrahedral description of the hyperbolicity cone could be constructed in a similar way, but this is currently purely speculative. In \S\ref{sec:bez} we point out how our construction is related to sum-of-squares decompositions of B\'ezout matrices and the construction in \cite{DetBez}.
 
Even in the original Dixon process for plane curves, details are somewhat subtle: For the construction to succeed as stated, the curve defined by $f$ must be smooth, and the existence of a contact curve satisfying the required genericity assumption (equivalent to the existence of a non-vanishing even theta characteristic) was not rigourously established until somewhat later. Additionally, the case of singular curves was, to our knowledge, only fully settled and explicitly stated by Beauville in 2000 \cite{beauville}. Likewise, in our generalized Dixon process, we need to treat degenerate cases with care and need some genericity assumptions.

Our generalized Dixon process has the additional feature that the size of the matrix $M$ decreases if the interlacer $g$ has real contact points with $f$. In particular, if $g$ is an interlacer with \textit{only real intersection points}, our statement reduces to that of the Helton--Vinnikov theorem. This leads us to the study of interlacers with real intersection (i.e.~contact) points. Such interlacers are necessarily on the boundary of the cone ${\rm Int}(f,e)$ of all interlacers of $f$. An extreme ray of that cone will necessarily have a certain number of real contact points (Lemma \ref{Lemma:ExtremalContact}). However, we do not know whether there always exists an interlacer with the maximal number $d(d-1)/2$ of real contact points. Even in the case $d=4$, we only obtain a partial answer to this question (see \S\ref{subsection:quartics}). There remain interesting (and easily stated) open questions concerning interlacing curves and the geometry of the interlacer cone.

\newpage
\section{Extremal interlacers}

Let $f\in\R[x,y,z]$ be homogeneous of degree $d$ and hyperbolic with respect to
$e=(0:0:1)$, with $f(e)>0$. Let $C=\sV_\C(f)$ be the plane projective
curve defined by $f$. We denote by $C(f,e)$ the closed hyperbolicity region of $f$ with respect to $e$ in the real projective plane.

\begin{Def}
  Let $f,g\in\R[t]$ be univariate polynomials with only real zeros and with $\deg(g)=\deg(f)-1$. Let
  $\alpha_1\leq\cdots\leq\alpha_d$ be the roots of $f$, and let
  $\beta_1\leq\cdots\leq\beta_{d-1}$ be the roots of $g$. We say that
  \emph{$g$ interlaces $f$} if $\alpha_i\le\beta_i\le\alpha_{i+1}$ holds for all
  $i=1,\dots,d-1$. If all these inequalities are strict, we say that
  \emph{$g$ strictly interlaces $f$}.

  If $f\in \R[x,y,z]$ is hyperbolic with respect to $e$ and $g$ is homogeneous of
  degree $\deg(f)-1$, we say that \emph{$g$ interlaces $f$ with
    respect to $e$} if $g(te+v)$ interlaces $f(te+v)$ for every
  $v\in\R^3$. This implies that $g$ is also hyperbolic with
  respect to $e$. We say that $g$ \emph{strictly interlaces $f$} if
  $g(te+v)$ strictly interlaces $f(te+v)$ for every $v\in\R^3$ not in $\R e$.
\end{Def}

With $f$ as above, let $g$ be any form in $\R[x,y,z]$
coprime to $f$. We say that an intersection point $p\in\sV_\C(f,g)$ is
a \emph{contact point} of $g$ with $f$ if the intersection
multiplicitity ${\rm mult}_p(f,g)$ is even.  If all intersection points are contact points, then $g$ is called a \emph{contact curve} of $f$. A \emph{curve of real contact} is a curve $g$ for which all real intersection points are contact points, without any assumption on non-real intersection points. Any interlacer is a curve of real
contact. 

Interlacers of $f$ appear naturally in the context of determinantal representations of $f$ (see \cites{PV13, us}). For example, if $f=\det(xA+yB+zC)$ is a real symmetric and definite determinantal representation of $f$, then every principal $(d-1)\times(d-1)$ minor of $xA+yB+zC$ is an interlacer of $f$ (see \cite[Thm. 3.3]{PV13}). Furthermore, such a minor defines a contact curve (see e.g.~\cite[Prop. 3.2]{PV13}). Conversely, given any interlacer of $f$ that is also a contact curve, one can construct a definite determinantal representation of $f$ and therefore a spectrahedral representation of its hyperbolicity region of size $d\times d$. However, for computational purposes, it is very difficult to actually find such an interlacer, even though its existence is guaranteed by the Helton--Vinnikov Theorem \cite{HV07}. In Section \ref{sec:GenDixon}, we will introduce a method for constructing from an arbitrary interlacer a spectrahedral representation of possibly larger size. We denote by
\[
{\rm Int}(f,e)=\bigl\{g\in\R[x,y,z]_{d-1}\suchthat g\text{ interlaces
}f\text{ and }g(e)>0\bigr\}
\]
the set of interlacers of $f$. It is shown in \cite[Cor. 2.7]{us} that this is a closed convex cone. Every boundary point of this cone has at least one contact point. In order to find interlacers with many contact points, it is therefore natural to consider extreme rays of this cone.

\begin{Def}
  Let $f$ be hyperbolic with respect to $e$.
  By an \emph{extremal interlacer} of $f$
  we mean an extreme ray of the cone ${\rm Int}(f,e)$.
\end{Def}

The next lemma gives a lower bound on the number of real contact points of an extremal interlacer.

\begin{Lemma}\label{Lemma:ExtremalContact} Assume that $f$ defines a smooth curve of degree $d$.
  Any extremal interlacer of $f$ has at least
\[
\left\lceil \frac{(d+1)d-2}{4}\right\rceil
\]
real contact points with $f$, counted with multiplicity.
\end{Lemma}

\begin{proof}
  Let $g$ be an extremal interlacer and let $k$ be the
  number of real contact points of $g$. By definition, the real part
  of the divisor $\div_C(g)$ is even, say $2D$, with $D$ real and
  effective of degree $k$. The space $V$ of forms $h$ of degree $d-1$
  with $\div_C(h)\ge 2D$ has dimension at least $n=(d+1)d/2-2k$ and
  contains $g$. If $n>1$, then $V$ contains another form $h$ linearly
  independent of $g$. We conclude that
  $g\pm\epsilon h\in{\rm Int}(f,e)$ for sufficiently small
  $\epsilon$. Thus $g$ is not extremal. Therefore, we must have
  $n\le 1$, which gives $k \ge \frac{(d+1)d-2}{4}$.
\end{proof}

\begin{Remark}
 For smooth $f$, given any $d-1$ real points on the curve, there is an extremal interlacer touching the curve in (at least) the given points. Indeed, it is clear from the above proof that it suffices to show that there is an interlacer passing through these $d-1$ points. The quadratic system of interlacers considered in \cite[Def. 3.1]{PV13} has dimension $d$, so we can prescribe $d-1$ points.
\end{Remark}

\begin{Remark}\label{rem:IrredExtremal}
	We do not know whether every hyperbolic curve possesses an {\it irreducible} extremal interlacer. This is true if $C$ is a smooth cubic: For any two distinct points $p$ and $q$ on $C$, there is an extremal interlacing conic $Q$  passing through $p$ and $q$, by the preceding remark. If $Q$ is reducible, it must factor into the two tangent lines to $C$ at $p$ and $q$. But $Q$ is a contact curve by Lemma \ref{Lemma:ExtremalContact}, hence the intersection point of the two tangents must lie on $C$. Clearly, this will not be the case for a generic choice of $p$ and $q$. This observation will be used at one point later on. It does not seem clear how to generalize this argument to higher degrees.
\end{Remark}

The following table shows the expected number of real contact points of an extremal interlacer compared with the
number of points for a full contact curve. 
\[
\begin{array}{c|cccccc}
  d                                           & 2 & 3 & 4 & 5 &6 & \cdots
  \\
  \hline
  \left\lceil \frac{(d+1)d-2}{4}\right\rceil  & 1 & 3 & 5 & 7 & 10
                                                              &\cdots\\
  \hline
  \frac{d(d-1)}{2}                             & 1 & 3 & 6 & 10 & 15 &\cdots
\end{array}
\]

An interlacer can have many more real contact points than the estimate given by Lemma \ref{Lemma:ExtremalContact}
and we do not know whether there is always one with only real intersection points.

\begin{Question}
 Does every hyperbolic plane curve have an interlacer that intersects the curve only in real points?
\end{Question}

Even without the interlacing condition, it seems to be unknown whether a real curve always possesses a real contact curve with only real contact points. In the case of plane quartic curves we have some partial answers to that question.

\subsection{The case of quartics}\label{subsection:quartics}
Let $C\subseteq\P^2$ be a smooth hyperbolic quartic that has a real bitangent touching $C$ in only real points. We will show that in this case there is a contact interlacer touching $C$ only in real points. It suffices to show that there is a conic touching both ovals in two real points. This, together with the above bitangent, will be the desired totally real interlacer.

Assume that $C(\R)$ is contained in the affine chart $z\neq 0$ (for smooth quartic curves this is not a restriction). Let $l\in\R[x,y]_1$ be a nonzero linear form. Maximizing and minimizing $l$ on the hyperbolicity region gives us two different linear polynomials $l_1$ and $l_2$ that are parallel and whose zero sets are tangent to the inner oval at some points $p_1$ and $p_2$ (see Figure \ref{fig:quartics}).

Choose the signs such that both $l_1$ and $l_2$ are nonnegative on the inner oval. We consider the pencil of conics whose zero sets pass through $p_1$ and $p_2$ such that the tangent lines of the conics at $p_1$ and $p_2$ are defined by $l_1$ and $l_2$ respectively. This pencil is given by $q_\lambda=g^2-\lambda l_1 l_2$, $\lambda\in\R$, where $g$ is the line spanned by $p_1$ and $p_2$. The zero set of $q_\lambda$ is completely contained in the interior of the outer oval for small $\lambda>0$. Label the two half spaces defined by $g$ by $1$ and $2$ and let $\lambda_i>0$ be the smallest positive number such that the zero set of $q_{\lambda_i}$ intersects the outer oval in the half-space labeled by $i$. We observe that both $q_{\lambda_i}$ have three real contact points with $C$.  If $\lambda_1=\lambda_2$, then we are done.

\begin{figure}[!ht]
\begin{subfigure}{.3\textwidth}
\centering
\includegraphics[width=4cm,height=4cm]{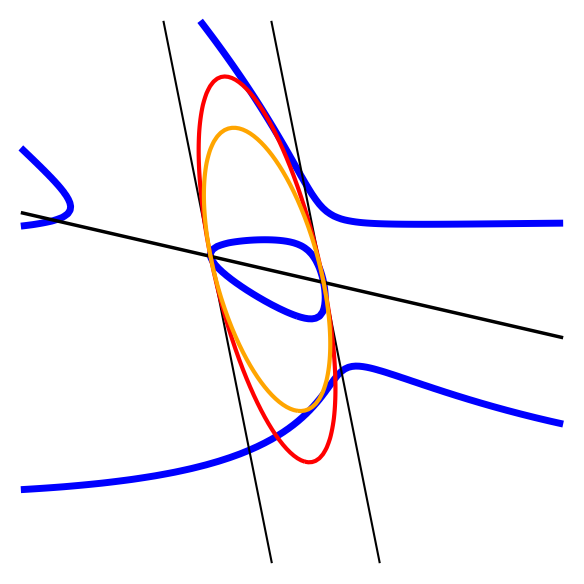}
\caption{}
\label{fig:quarticsA}
\end{subfigure}\hfill
\begin{subfigure}{.3\textwidth}
\centering
\includegraphics[width=4cm,height=4cm]{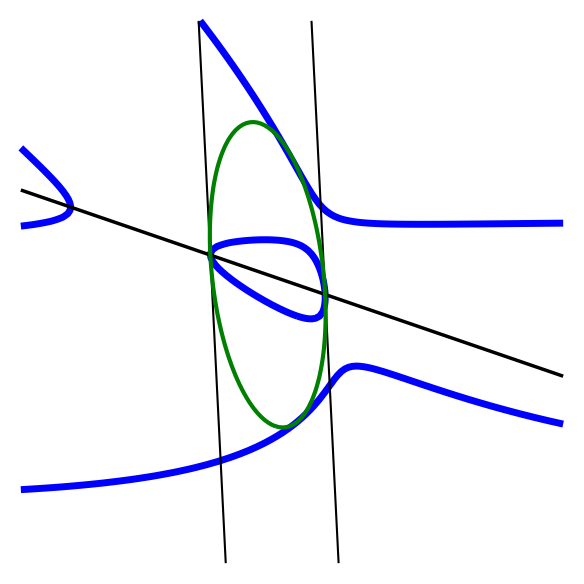}
\caption{}
\label{fig:quarticsB}
\end{subfigure}\hfill
\begin{subfigure}{.3\textwidth}
\centering
\includegraphics[width=4cm,height=4cm]{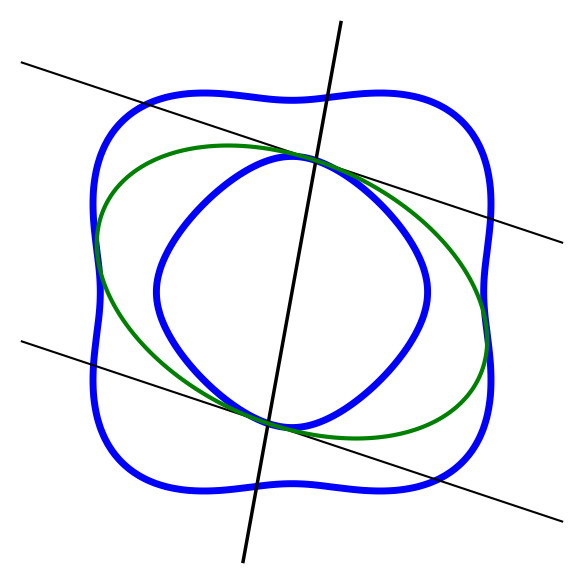}
\caption{}
\label{fig:quarticsC}
\end{subfigure}
\caption{Quadrics touching hyperbolic quartics in real points}
\label{fig:quartics}
\end{figure}

Now we let the linear form $l$, which we started with, vary continuously and we also keep track of the labels of the half-spaces in a continuous manner. The resulting conic $q_{\lambda_1}(l)$ depends continuously on $l$ and we note that $q_{\lambda_1}(-l)=q_{\lambda_2}(l)$. Note that one of the zero sets of $q_{\lambda_1}(l)$ resp. $q_{\lambda_1}(-l)$ on $C$ contains a pair of complex conjugate points (the orange oval in Figure \eqref{fig:quarticsA}) whereas the other one contains only real points of $C$ (the red oval in Figure \eqref{fig:quarticsA}). Therefore, there must be a linear form $l_0$ such that $q_{\lambda_1}(l_0)$ has the desired properties (Figure \eqref{fig:quarticsB} and \eqref{fig:quarticsC}).

If there is no bitangent touching the quartic in two real points, we do not know whether there always exists an interlacer intersecting the curve in only real points. The next example shows that this is at least sometimes the case.
\begin{Example}
 We consider the smooth plane quartic defined by \begin{eqnarray*}f&=&1250000 x^4-1749500 x^3 y-2250800 x^2 y^2-4312500 x^2 z^2\\&&+69260 x y^3+786875 x y z^2+88176 y^4+1141000 y^2 z^2+1687500 z^4.\end{eqnarray*} Its real locus consists of two nested ovals both of which are convex (Figure  \ref{fig:6points}), meaning that there is no bitangent touching the curve in two real points. Nevertheless the interlacer given by $$g=500 x^3-800 x^2 y-740 x y^2-625 x z^2+176 y^3+1000 y z^2$$ intersects the quartic curve only in real points. Indeed, its divisor is given by: $$4\cdot (4:-5:0)+2\cdot(11:5:0)+2\cdot(1:5:0)+2\cdot(7:10:-10)+2\cdot(7:10:10).$$
 \begin{figure}[h]
     \includegraphics[width=7cm]{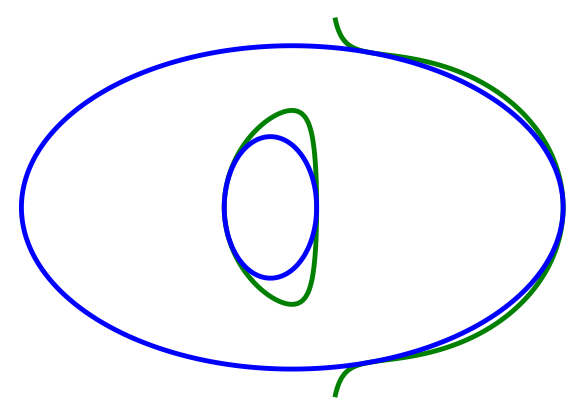}
 \caption{A hyperbolic quartic curve (in blue) and a cubic interlacer (in green) with only real intersection points}
 \label{fig:6points}
 \end{figure}
\end{Example}

\section{A generalized Dixon process}\label{sec:GenDixon}

Given a real hyperbolic form $f$ of degree $d$ and an interlacer $g$ of degree $d-1$, we wish to produce a real symmetric determinantal representation of $f$ with a principal minor divisible by $g$. If $g$ is a contact curve, this is achieved through the classical Dixon process. We will extend the procedure in such a way that the resulting representation will reflect any real contact points between $f$ and $g$, relating to our discussion of extremal curves of real contact in the previous section.

Let $f$ be irreducible and hyperbolic with respect to $e\in\P^2(\R)$ and assume that the plane curve $\sV_\C(f)$ is smooth. Let $g$ be an interlacer of $f$ with $r$ real contact points $p_1,\dots,p_r$, counted with multiplicities. Consider the $d(d-1)-2r$ further intersection points, which are non-real and therefore come in complex conjugate pairs, say $q_1,\dots,q_s,\ol q_1,\dots,\ol q_s$, so that $d(d-1)=2r+2s$. For each $i=1,\dots,s$ let $\ell_i$ be a linear form defining the unique (real) line joining $q_i$ and $\ol{q_i}$. We will make the following assumptions:

\begin{itemize}
\item[(G1)] No three of the intersection points of $f$ with $g$ lie on a line.
\item[(G2)] No three of the $\ell_i$ pass through the same point.
\item[(G3)] $f$ does not vanish on any point where two of the $\ell_i$ intersect.
\end{itemize}
We begin by showing that such an interlacer always exists.

\begin{Lemma}
 There exists a strict interlacer for which the genericity assumptions (G1), (G2), and (G3) are satisfied.
\end{Lemma}

\begin{proof}
 Every choice of $k=\frac{1}{2}d(d+1)-1$ points on the zero set of $f$ that pose linearly independent conditions on forms of degree $d-1$ determines a unique such form. The other zeros of this $(d-1)$-form on the zero set of $f$ depend continuously on the choice of the $k$ points. By the General Position Theorem \cite[Ch. III, \S 1]{arba}, any neighborhood of the given interlacer contains a strict interlacer $g$ with the property that its zero set intersects the one of $f$ in $d(d-1)$ distinct points any $k$ of which pose linearly independent conditions on forms of degree $d-1$. Then we can slightly perturb any subset of $k$ points in this intersection, and thus $g$, so that the number of triples of points in the intersection that lie on a line decreases. Thus we can find a strict interlacer of $f$ with the property that no three intersection points with the zero set of $g$ lie on a line, so that genericity condition (G1) is satisfied. By the same argument, we can satisfy condition (G3). 
 
 For condition (G2), we need to move six points spanning three of the lines. Thus the same argument applies, provided that $k\ge 6$, which means $d>3$. The case $d\le 2$ being trivial, we are left with condition (G2) for cubics ($d=3$). In this case, we argue as follows: Suppose there is no interlacing conic satisfying condition (G2). Since the condition is Zariski-open, this would imply that condition (G2) is violated for any conic, strictly interlacing or not. But
 Lemma \ref{Lemma:ExtremalContact} and the subsequent Remark \ref{rem:IrredExtremal} imply that there exists an irreducible conic $g$ touching $f$ in three real points. Considering $g$ as the limit of forms all of whose intersection points with $f$ are simple, the assumption will imply that the three tangents to $\sV(g)$ at the contact points meet in one point. But since $g$ is irreducible of degree $2$, this is impossible. This contradiction shows the claim.
\end{proof}
 
Under these genericity assumptions, we will construct a symmetric linear determinantal representation $M$ of $\ell_1\cdots \ell_s\cdot f$ such that $\sS(M)$ is the hyperbolicity region of $f$. Furthermore, the interlacer $g$ divides a principal minor of $M$. The main result of this section is as follows:

\begin{Thm}\label{Thm:GenDixon}
  Let $f$ be an irreducible form of degree $d$ that is hyperbolic with respect to $e\in\P^2(\R)$ and assume that the plane curve $\sV(f)$ is smooth. Let $g$ be an interlacer of $f$ with $r$ real contact points, counted with multiplicities, that satisfies the genericity assumptions (G1), (G2), (G3). Then there exists a symmetric linear matrix
  pencil $M$ of size
\[
   m=\frac{d^2+d-2r}{2}
\]
  which is positive definite at $e$ and such
  that $C(f,e)=\sS(M)$. We can
  choose $M$ in such a way that $g$ divides the principal minor $M_{1,1}$ of $M$ and $\det(M)/f$ is a product of $m-d$ linear
  forms. Furthermore, each $(m-1)\times (m-1)$ minor $M_{1l}$, $1\leq l\leq m$, of $M$ is also divisible by the product of these $m-d$ linear forms.
\end{Thm}

The proof will consist of an algorithm that produces the desired representation given $f$ and $g$. 

We begin with some preliminaries. Given any two real ternary forms $f,g$ of degree $d$ resp.~$d'$ without common components, we denote by $(f.g)$ the \emph{intersection cycle} of $f$ and $g$, consisting of the intersection points of the curves $\sV(f)$ and $\sV(g)$ in $\P^2(\C)$. It is a $0$-cycle, i.e.~an element of the free abelian group over the points of $\P^2(\C)$. Explicitly, $(f.g)=\sum_{i=1}^r m_ip_i$, with $\sV(f)\cap\sV(g)=\{p_1,\dots,p_r\}$ and $m_i$ positive integers, the intersection multiplicities. By B\'ezout's theorem, we have $\sum_{i=1}^r m_i=dd'$. Intersection cycles are additive, i.e.~$((f_1\cdot f_2).g)=(f_1.g)+(f_2.g)$. Furthermore, there is a natural partial order on $0$-cycles, by comparing coeffcients. We need the following classical result from the theory of plane curves, which we restate in the form we require.

\begin{Thm}[Max Noether]
        Let $f,g,h$ be real ternary forms. Assume that $f$ is irreducible and does not divide $gh$, and that the curve $\sV(f)\subset\P^2(\C)$ is smooth. If $(h.f)\ge (g.f)$, then there exist real forms $a$ and $b$ such that 
	\[
	h=af+bg.
	\]
\end{Thm}

\begin{proof}
	See \cite[\S5.5, Prop.~1]{max}.
\end{proof}

  Now let $f$ and $g$ be given as in the statement of Thm.~\ref{Thm:GenDixon}, with intersection points $p_1,\dots,p_r,q_1,\dots,q_s,\ol{q_1},\dots,\ol{q_s}$ as before, and let $\ell_i$ be the linear form defining the line between $q_i$ and $\ol{q_i}$, for $i=1,\dots,s$, under the genericity assumptions (G1)--(G3).

  Put $h=\ell_1\cdots \ell_s$ and consider the polynomial
  $fh$. It is of degree $(d^2+d-2r)/2=m$ and hyperbolic with respect
  to $e$. Furthermore, since each line $\ell_i$ meets $C$ in the
  non-real point $q_i$, none of the lines pass through $C(f,e)$,
  so that $C(fh,e)=C(f,e)$. 
  
  It therefore suffices to construct a symmetric linear
  determinantal representation of $fh$ which is definite at $e$. This
  can be carried out with a modification of Dixon's method, which we now describe in several steps. 
  
  1)  Let $V$ be the linear space of real forms of degree $d-1$ vanishing at $p_1,\dots,p_r$. We have $\dim(V)\geq \frac{(d+1)d}{2}-r= d+s$, and we pick linearly independent forms $a_1,\dots,a_{d+s}\in V$, with $a_1=g$. We introduce names for all the occuring intersection points: 
  \begin{align*}
  (a_1.f)&=(g.f)=2\sum_{j=1}^r p_j + \sum_{j=1}^s (q_j+\ol{q_j})\\
  (a_i.f)&=\sum_{j=1}^r p_j + \sum_{j=1}^{r+2s} p_{ij}\ \ \ \text{ for }i\ge 2\\
  (\ell_i.f)&=q_i+\ol{q_i}+\sum_{j=1}^{d-2} r_{ij}\\
  (\ell_i.\ell_j) &= s_{ij}\ \ \ \text{ for }i\neq j
  \end{align*}
  2) Fix $k,l\in\{2,\dots,d+s\}$ with $k\le l$. We wish to find a real form $b_{kl}$ of degree $d+s-1$ such that
  \begin{align}\label{eqn:bonf}
  b_{kl}g - h a_ka_l \in (f)
  \end{align}
  by applying Max Noether's theorem: We compute the intersection cycles
  \begin{align*}
  (ha_ka_l.f)&=2\sum_{j=1}^r p_j+\sum_{j=1}^s (q_j+\ol{q_j}) + \sum_{j=1}^s\sum_{j'=1}^{d-2} r_{jj'}+\sum_{j=1}^{r+2s} p_{kj}+\sum_{j=1}^{r+2s} p_{lj}\\
  (g.f)&=2\sum_{j=1}^r p_j+\sum_{j=1}^s (q_j+\ol{q_j})
  \end{align*}
  and thus find $b_{kl}$ with
  \[
  (b_{kl}.f)=\sum_{j=1}^s\sum_{j'=1}^{d-2} r_{jj'}+\sum_{j=1}^{r+2s} p_{kj}+\sum_{j=1}^{r+2s} p_{lj}.
  \]
  3) Assume that $k=l$. Then we will produce a real form $q$ of degree $s-1$ such that $c_{kk}:=b_{kk}+qf$ satisfies
  \[
  (c_{kk}.\ell_i)=(b+qf.\ell_i) = \sum_{j=1}^{d-2} r_{ij} + \sum_{j\neq i}s_{ij} + 2t_{ki}
  \]
  for some real point $t_{ki}\in\ell_{ki}$, for all $i=1,\dots,s$.
  To this end, we let $\ell_0$ be a linear form which does not vanish on any of the $s_{ij}$. Let $h_{ij}=\frac{\ell_0\cdots\ell_s}{\ell_i \ell_j}$ and $\alpha_{ij}=-\frac{b_{kk}(s_{ij})}{h_{ij}(s_{ij})f(s_{ij})}$ for $1\leq i<j\leq s$. Note that $h_{ij}$ vanishes on all $s_{mn}$ except for $s_{ij}$. After replacing $b_{kk}$ by $b_{kk}+\sum_{i,j} \alpha_{ij} h_{ij} f$, we can thus assume that $b_{kk}$ vanishes on all the $s_{ij}$.
  
  Next, we consider
  \[
  q_\alpha=\sum_{j=1}^s \alpha_j\frac{\ell_1\cdots\ell_s}{\ell_j}.
  \]
  with $\alpha_1,\dots,\alpha_s\in\R$. The form $q_\alpha$ satisfies $q_\alpha(s_{ij})=0$ for all $j\neq i$ for any choice of the $\alpha_j$. If we now take $q=\wt q+q_\alpha$, we find
  \[
  (b_{kk}+qf.\ell_i) = \sum_{j=1}^{d-2} r_{ij} + \sum_{j\neq i}s_{ij} + u_i+v_i
  \]
  with $u_i$ and $v_i$ depending on $\alpha$. Restricting to $\ell_i$ we therefore get $b_{kk}+q f=P\cdot(\tilde{b}+\alpha_i \tilde{f})$ where $P$ is a nonzero polynomial whose roots are the $r_{ij}$ and $s_{ij}$, and where $\tilde{b}$ and $\tilde{f}$ are polynomials of degree two. After possibly replacing $\alpha_i$ by its negative, we can assume that $\tilde{f}$ is strictly positive on $\ell_i$ since it has no real zeros on $\ell_i$. Therefore, we can choose $\alpha_i$ in such a way that $\tilde{b}+\alpha_i \tilde{f}$ has a double zero $t_{ki}$ and that makes the product of $b_{kk}+q f$ and $f\cdot\frac{\ell_1\cdots\ell_s}{\ell_i}\ell_i(e)$ nonnegative on $\ell_i$. The reasons for the latter requirement will become clear in a later step.

  4) Similarly, if $k<l$, we can find a real form $q$ of degree $s-1$ such that $c_{kl}:=b_{kl}+qf$ satisfies
  \[
  (c_{kl}.\ell_i)=(b_{kl}+qf.\ell_i) = \sum_{j=1}^{d-2} r_{ij} + \sum_{j\neq i}s_{ij} + t_{ki}+t_{ki}'
  \]
  for some real point $t_{ki}'\in\ell_i$. In fact, we even have that $t_{ki}'=t_{li}$. Indeed, this follows from (\ref{eqn:bonf}) and the following lemma applied to each $\ell_i$.
  \begin{Lemma}
   Let $f\in\R[t]$ be a polynomial of degree two without real zeros. Let $a,b,c\in\R[t]$ be polynomials of degree at most two such that $a$ and $c$ both have a double zero, $ac$ is nonnegative and $b$ vanishes at the zero of $a$. If $ac=b^2 \mod f$, then $b$ vanishes at the zero of $c$ as well.
  \end{Lemma}
\begin{proof}
Let $a=\alpha(t-\beta)^2$, $c=\alpha'(t-\beta')^2$ and $b=\gamma(t-\beta)(t-\beta'')$ for some $\alpha,\alpha',\beta,\beta',\beta'',\gamma\in\R$ with $\alpha\alpha'\geq0$. We have by assumption $$\alpha\alpha'(t-\beta)^2(t-\beta')^2=\gamma^2(t-\beta)^2(t-\beta'')^2 \mod f.$$ Since $\R[t]/(f)$ is isomorphic to the field of complex numbers, it follows that $$\alpha\alpha'(t-\beta')^2=\gamma^2(t-\beta'')^2 \mod f.$$ If $\gamma\neq0$, then $\alpha\alpha'>0$ and $t-\beta'=\pm \sqrt{\frac{\gamma^2}{\alpha\alpha'}}(t-\beta'')\mod f$. Finally, it follows that $\alpha\alpha'=\gamma$ and that $\beta'=\beta''$ because $1,t\in\R[t]/(f)$ are $\R$-linearly independent.
\end{proof}

  If $k>l$, we let $c_{kl}=c_{lk}$.
  
  5) We now put $c_{1k}=c_{k1}=ha_k$ and consider the matrix $N$ with entries $c_{kl}$, for $k,l=1,\dots,d+s$. By construction, the $2\times 2$-minors
  \[
  c_{11}c_{kl}-c_{1k}c_{1l} = hgc_{kl}-h^2a_ka_l = h(gc_{kl}-ha_ka_l)
  \]
  are divisible by $fh$. Since the first row of $N$ is not divisible by $f$, it follows that all $2\times 2$-minors of $N$ are divisible by $f$. We need to show that all $2\times 2$-minors $c_{kl}c_{k'l'}-c_{kl'}c_{k'l}$
  are also divisible by $h$. Let $u$ be such a minor and fix $i\in\{1,\dots,s\}$. Note that $u$ has degree $2d+2s-2$ and vanishes (with multiplicities) on the $2d+2s-2$ points $2\sum_{j=1}^{d-2} r_{ij}$, $2\sum_{j\neq i}s_{ij}$, $(t_{ki}+t_{k'i}+t_{li}+t_{l'i})$ on $\ell_i$, since both products $c_{kl}c_{k'l'}$ and $c_{kl'}c_{k'l}$ vanish at those points. Since $u$ is divisible by $f$, it also vanishes at $q_i+\ol{q_i}$. Thus $u$ vanishes identically on $\ell_i$ for each $i$, which implies $h|u$.
  
  6) In this step we show that $c_{22}$ interlaces $fh$. This can be done by proving that $c_{22}\cdot \textrm{D}_e(fh)$ is nonnegative on the zero set of $fh$ \cite[Thm. 2.1]{us}.  Here $\textrm{D}_e(fh)$ denotes the derivative of $fh$ in direction $e$. We have $$\textrm{D}_e(fh)=h\cdot \textrm{D}_e f+f\cdot \sum_{i=1}^s\ell_i(e) \frac{\ell_1\cdots\ell_s}{\ell_i}.$$ We can rewrite this modulo $f$ and find 
  $$c_{22}\cdot \textrm{D}_e(fh)=c_{22} \cdot h\cdot \textrm{D}_e f=\frac{h a_2^2}{g} \cdot h\cdot \textrm{D}_e f=\frac{\textrm{D}_e f}{g} h^2 a_2^2\quad ({\rm mod}\ f)$$ 
  by (\ref{eqn:bonf}). This is nonnegative on the zero set of $f$ because both $\textrm{D}_e f$ and $g$ are interlacers. On the other hand, modulo $\ell_i$ we obtain
  $$c_{22}\cdot \textrm{D}_e(fh)=c_{22}\cdot \ell_i(e)\cdot \frac{\ell_1\cdots\ell_s}{\ell_i}\quad ({\rm mod}\ \ell_i) $$ 
  which is nonnegative on the line defined by $\ell_i$ by the choices made in Step $3)$.

  7) Now we proceed as in the usual Dixon process, referring to \cite{PV13} for details: Since all $2\times 2$-minors of the $(d+s)\times (d+s)$-matrix $N$ are divisible by $fh$, its maximal minors are divisible by $(fh)^{d+s-2}$ (see for example \cite[Lemma 4.7]{PV13}). The signed maximal minors of $N$ have degree $(d+s-1)^2$ and are the entries of the adjugate matrix $N^{\rm adj}$. It follows that 
  \[
  M=(fh)^{2-d-s}\cdot N^{\rm adj}
  \]
   has linear entries. Using the familiar identity $NN^{\rm adj}=\det(N)\cdot I_{d+s}$, we conclude $$\det(M)=\gamma\cdot fh$$ 
   for some constant $\gamma\in\R$. It remains to show that $\gamma\neq 0$. Suppose $\gamma=0$, then $\det(M)$ is identically zero, hence so is $\det(N)$. In particular, the matrix $N(e)$ is singular. Let $\lambda\in\R^{d+s}$ be a non-trivial vector in the kernel of $N(e)$ and consider the polynomial $\wt g=\lambda^T N{\lambda}$. It follows from the linear independence of the entries of the first row of $N$ that $\wt g$ is not the zero polynomial \cite[Lemma 4.8]{PV13}. Since $c_{22}$ interlaces $fh$ by (6), so does $\wt g$ \cite[Thm.~3.3, (1)$\Rightarrow$(2)]{PV13}, contradicting $\wt g(e)=0$. That $M(e)$ is definite also follows from the fact that $c_{22}$ interlaces $fh$, by \cite[Thm.~3.3, (2)$\Rightarrow$(3)]{PV13}. Note that the result in \cite{PV13} is stated only for irreducible curves. However, the same argument will apply here, since we have shown that $c_{22}$ is coprime to $fh$ (unlike $c_{11}$, which is divisible by $h$).

\bigskip
\noindent This finishes the construction of the determinantal representation $M$ of $fh$. Finally, we note that the spectrahedron $\mathcal{S}(M)$ coincides with the hyperbolicity region $C(f,e)$ of $f$. Since $\det(M)=f\cdot\ell_1\cdots\ell_s$, this simply amounts to the fact that the lines $\ell_1,\dots,\ell_s$ do not pass through $C(f,e)$. Indeed, each $\ell_j$ has two non-real intersection points with $C$, while lines passing through the hyperbolicity region will meet $C$ in only real points. This completes the proof of Theorem \ref{Thm:GenDixon}.
\bigskip

\begin{Remark}
 Clearly, the corank of the constructed matrix pencil $M$ is at least one at each point where $fh$ vanishes. It can have corank more than one only at singularities of $fh$, i.e.~in our case the points where two components intersect. Since the adjugate $N=M^{\rm adj}$ vanishes identically at the points $r_{ij}$ and $s_{ij}$ and because these are ordinary nodes, the corank of $M$ at these points is exactly two. On the other hand, we have constructed $N$ in such a way that it is not entirely zero at the points $q_j$ and $\ol{q_j}$. Thus $M$ has corank one at these points. This shows in particular that $M$ is not equivalent to a block diagonal matrix with more than one block.
\end{Remark}

\begin{Remark}
 The vector space $V$ in Step $1)$ of our construction can be found without computing all the real contact points $p_1,\ldots,p_r$. Indeed, by genericity assumption (G1) the $q_i, \ol{q_i}$ are all simple intersection points. Therefore, the $p_i$ can be computed as the singular locus of the zero dimensional scheme cut out by $f$ and $g$ via the Jacobian criterion.
\end{Remark}

Next we observe that the genericity assumption in the theorem, as well as the smoothness assumption on $f$, can be dropped for strict interlacers by applying a limit argument.

\begin{Cor}\label{Kor:GenDixon}
  Let $f$ be a real form of degree $d$ that is hyperbolic with respect to $e\in\P^2(\R)$, and let $g$ be a strict interlacer of $f$. Then there exists a symmetric linear matrix
  pencil $M$ of size $\frac{d^2+d}{2}$ which is definite at $e$ and such
  that $C(f,e)=\sS(M)$. We can
  choose $M$ in such a way that $g$ divides a principal minor of $M$ and $\det(M)/f$ is a product of $\frac{d^2-d}{2}$ linear
  forms.
\end{Cor}

\begin{proof}
	Let $m=\frac{d^2+d}{2}$. We may assume that $f(e)=1$ and consider only monic representations $f=\det(M)$, i.e.~with $M(e)=I_m$. The determinant map taking a monic symmetric real linear matrix pencil of size $m\times m$ to its determinant is proper, hence its image is closed (see for example \cite[Lemma 3.4]{PV13}). If $g$ is a strict interlacer of $f$, the pair $(f,g)$ is in the closure of the set of pairs $(\wt f,\wt g)$, where $\wt f$ is hyperbolic with respect to $e$, $\sV(\wt f)$ is smooth, and $\wt g$ is a strict interlacer of $\wt f$ satisfying the genericity assumptions (G1)--(G3). Therefore, there exists a sequence $(\wt f_n,\wt g_n)$ converging to $(f,g)$ together with representations $\wt f_n=\det(\wt M_n)$ with $\wt g_n$ dividing the first principal minor of $\wt M$ and $\det(\wt M)/\wt f$ a product of $m-d$ linear forms, by Theorem \ref{Thm:GenDixon}. The sequence $\wt M_n$ then has a subsequence converging to a matrix pencil $M$, which is the desired determinantal representation of $f$.
\end{proof}

\begin{Remark}
	The procedure of approximating a given hyperbolic form together with an interlacer as in the proof above may be difficult to carry out in practice. However, the generalized Dixon process can often be applied (with small modifications if needed) even when the genericity assumptions fail.
\end{Remark}

As a further consequence, we can prove the following rationality result.

\begin{Thm}\label{thm:rationalspec}
 Let $f\in\Q[x,y,z]_d$ be a polynomial hyperbolic with respect to $e\in\R^3$ whose real projective zero set is smooth. Then its hyperbolicity cone is of the form $$\{(x,y,z)\in\R^3:\, xA+yB+zC\succeq0\}$$where $A,B,C$ are symmetric matrices with rational entries.
\end{Thm}

\begin{proof}
 Let $m\in\Q[x,y,z]_{d-1}^N$ be the vector of all monomials of degree $d-1$ and let $N=\binom{d+1}{2}$. The equation \begin{align}\label{eqn:rat} (xA+yB+zC)\cdot m=f\cdot v\end{align} poses linear conditions on the entries of the symmetric $N\times N$ matrices $A,B,C$ and on the entries of $v\in\R^N$. These linear conditions are defined over the rational numbers. Applying the above construction gives a solution to this system of linear equations with $e_0 A+e_1 B+e_2 C$ positive definite and $\det(xA+yB+zC)=h\cdot f$ where $h$ is a product of linear forms whose zero set does not intersect the hyperbolicity cone of $f$. Since the rational solutions to (\ref{eqn:rat}) are dense in the solution set over the real numbers, we can find rational matrices $A,B,C$ satisfying (\ref{eqn:rat}) with $e_0 A+e_1 B+e_2 C$ being positive definite, as well. Then $\det(xA+yB+zC)$ is not the zero polynomial and is divisible by $f$, since the pencil has a nonzero kernel vector whenever $f$ vanishes at $(x,y,z)$ by (\ref{eqn:rat}). If $A,B,C$ are chosen close enough to our original solution, the other factor of $\det(xA+yB+zC)$ will not intersect the hyperbolicity cone of $f$ either.
\end{proof}

The next example shows that the smallest size of a rational spectrahedral representation is in general larger than the degree of the curve.

\begin{Example}
\label{ex:rational}
  Consider the univariate polynomial $p=x^3 - 6 x - 3\in\Q[x]$. It has three distinct real zeros but is irreducible over the rational numbers by Eisenstein's criterion. The plane elliptic curve defined by $y^2=p(x)$ is hyperbolic. Its hyperbolicity cone has the following spectrahedral representation with rational $4\times4$ matrices: $$\{(x,y,z)\in\R^3:\,\,\begin{pmatrix} 3 z& y& -x - z& -3 x + z\\y& -x + 2 z& 0& -y\\-x - z&  0& z& x + 4 z\\-3 x + z& -y& x + 4 z& -x + 18 z \end{pmatrix}\succeq0\}. $$ This was obtained by applying our construction to the interlacer $y^2 + 3 x z + z^2$ with two real contact points (Figure \ref{fig:rational}).

\begin{figure}[!ht]
\centering
\includegraphics[width=8cm,height=5cm]{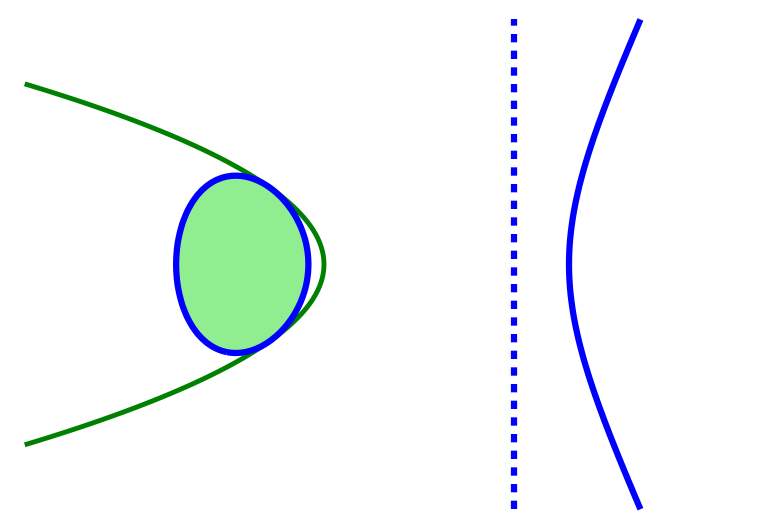}
\caption{Hyperbolic cubic (in blue), an interlacer touching in 2 real points (in green) and the linear factor (dashed in blue)}
\label{fig:rational}
\end{figure}

  It also has a $3\times3$ spectrahedral representation with real matrices by the Helton--Vinnikov Theorem. It does, however, not have such a representation with rational $3\times3$ matrices. Indeed, any such representation would yield a contact interlacer defined over the rational numbers by taking some principal $2\times2$ minor. This interlacer would give rise to a divisor $D$ defined over the rational numbers with $2D=6P_\infty$ where $P_\infty$ is the point of the curve at infinity. Thus $D-3P_\infty$ would be an even theta characteristic defined over the rationals. On the other hand, the three even theta characteristics of the curve are given by $P_i-P_\infty$ for $P_1,P_2,P_3$ the three intersection points of the curve with the $x$-axis. These are clearly not defined over the rationals.
\end{Example}

\subsection{B\'ezout matrices}\label{sec:bez}
Let $f, g \in \R[t]$ be two univariate polynomials having degrees $\deg(f)=d$ and $\deg(g) = d-1$. The \textit{B\'ezout matrix} of $f$ and $g$ is defined as follows. We write \[\frac{f(s)g(t)-f(t)g(s)}{s-t}=\sum_{i,j=1}^d b_{ij} s^{i-1} t^{j-1} \] for some real numbers $b_{ij}$. Then the B\'ezout matrix is defined as $\textnormal{B}(f,g)=(b_{ij})_{ij}$. Note that $\textnormal{B}(f,g)$ is always a real symmetric matrix. The B\'ezout matrix can be used to detect the properties of being real-rooted and interlacing.

\begin{Thm}[see \S2.2 of \cite{Krein81}]{\label{thm:bezuni}}
 Let $f,g \in \R[t]$ be univariate polynomials with $d=\deg(f)=\deg(g)+1$. Then the following are equivalent:
 \begin{enumerate}[(i)]
  \item The B\'ezout matrix $B(f,g)$ is positive semidefinite.
  \item The polynomial $g$ interlaces $f$.
 \end{enumerate}
 Furthermore, the B\'ezout matrix has full rank if and only if $f$ and $g$ have no common zero.
\end{Thm}

In the multivariate case we can proceed analogously. Let $f,g\in\R[x_0,\ldots,x_n]$ be homogeneous polynomials of degrees $d$ and $d-1$ respectively. We assume that $f$ and $g$ do not vanish at $e=(1,0,\ldots,0)$. Then, writing $x=(x_1,\ldots,x_n)$, we have \[\frac{f(s,x)g(t,x)-f(t,x)g(s,x)}{s-t}=\sum_{i,j=1}^d b_{ij} s^{i-1} t^{j-1} \] for some homogeneous polynomials $b_{ij}\in\R[x_1,\ldots,x_n]$ of degree $2d-(i+j)$. Again, we define the B\'ezout matrix as $\textnormal{B}(f,g)=(b_{ij})_{ij}$. It follows from the above theorem that $\textnormal{B}(f,g)$ is positive definite for every $0\neq x\in\R^n$ if and only if $f$ is hyperbolic with respect to $e$ and $g$ is a strict interlacer of $f$.

\begin{Remark}\label{rem:wron}
 The B\'ezout matrix $\textnormal{B}(f,g)$ is closely related to the \textit{Wronskian polynomial} $\textnormal{W}(f,g)=\textnormal{D}_ef \cdot g-f\cdot\textnormal{D}_eg$. Namely, if we let $w=(1,x_0,\ldots,x_0^{d-1})^t$, then $\textnormal{W}(f,g)=w^t\cdot\textnormal{B}(f,g)\cdot w$. Indeed, by the definition of the B\'ezout matrix the right-hand side equals $$\lim_{s\to t}\biggl( \frac{f(s,x)g(t,x)-f(t,x)g(s,x)}{s-t}\biggr)=\textnormal{W}(f,g).$$ We also note that for square-free polynomials $f$ the polynomial $g$ of degree $\deg(f)-1$ is uniquely determined by $\textnormal{W}(f,g)$.
\end{Remark}

We can use the Wronskian polynomial $\textnormal{W}(f,g)$ to describe the set
${\rm Int}(f,e)$ of interlacers of $f$ in direction $e$, which is a convex cone.
By \cite[Cor.2.7]{us}, ${\rm Int}(f,e)$ can be represented as a linear image of
a section of the cone of positive polynomials of degree $2d-2$, where $d=\deg f$:
\[
{\rm Int}(f,e)=\bigl\{g\in\R[x,y,z]_{d-1} \suchthat \textnormal{W}(f,g) \geq 0\bigr\}.
\]
Whenever $\textnormal{W}(f,g)$ is a sum of squares, the cone ${\rm Int}(f,e)$ can be sampled
by solving a linear matrix inequality as shown in the following example.

\begin{Example}
  The cubic $f=x^3+2x^2y-xy^2-2y^3-xz^2$ is hyperbolic with respect
  to $e=(1,0,0)$, and $C(f,e)$ is the green region in
  Figure \ref{expExact}.

  \begin{figure}[ht!]
  \centering
  \begin{tabular}{c}
    \includegraphics[width=8cm]{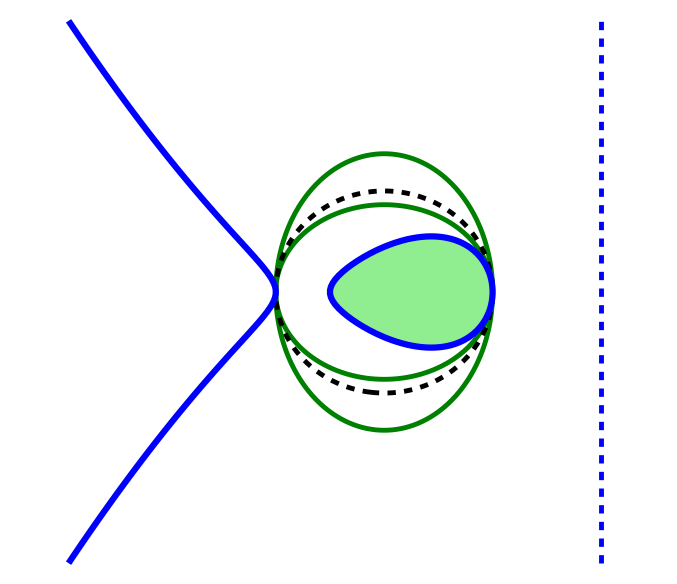}
  \end{tabular}
  \caption{A cubic hyperbolic curve (in blue) with three interlacers, one defined over $\Q$ (in dashed black) and two over an extension of degree $4$ (in green). The dashed blue line is the extra factor in the determinantal representation.}
  \label{expExact}
  \end{figure}

  Let $g = x^2 + g_{110}xy+g_{101}xz+g_{020}y^2+g_{011}yz+g_{002}z^2$
  be a generic quadratic form such that $g(e)=1$. 
  The Wronskian of $f,g$ in direction $e$ is the ternary quartic
  \begin{align*}
  \textnormal{W}(f,g) & = 2g_{110}x^3y+2g_{110}x^2y^2+2g_{110}y^4+2g_{101}x^3z+2g_{101}x^2yz+2g_{101}y^3z+ \\
& +3g_{020}x^2y^2+ 4g_{020}xy^3-g_{020}y^4-g_{020}y^2z^2+3g_{011}x^2yz+4g_{011}xy^2z- \\
  & -g_{011}y^3z-g_{011}yz^3+3g_{002}x^2z^2+ 4g_{002}xyz^2-g_{002}y^2z^2-g_{002}z^4+ \\
  & +x^4+x^2y^2+x^2z^2+4xy^3.
  \end{align*}
  Let $G=(G_{ij})$ be a symmetric $6 \times 6$ matrix of unknowns, and consider the linear
  system $\textnormal{W}(f,g) = m^t \cdot G \cdot m$, where $m$ is the vector of
  monomials of degree 2 in $x,y,z$. We obtain that $G$ (the {\it Gram matrix} of
  $\textnormal{W}(f,g)$, cf. \cite{powers1998algorithm}) has the form
\begin{center}
\label{gramCubic}
$G=$
\scalebox{0.65}{
$
\left[
\begin{array}{cccccc}
1 & g_{110} & g_{101} & G_{14} & g_{101}-G_{23}+\frac{3}{2}g_{011} & G_{16} \\
g_{110} & 3 g_{020} + 2 g_{110} + 1 - 2 G_{14} & G_{23} & 2g_{020}+2 & -G_{34}+2g_{011} & -G_{35}+2g_{002} \\
g_{101} &  G_{23} & 1+3g_{002}-2G_{16} & G_{34} & G_{35} & 0 \\
G_{14} & 2g_{020}+2 & G_{34} & 2g_{110}-g_{020} & g_{101}-\frac{1}{2}g_{011} & G_{46} \\
g_{101}-G_{23}+\frac{3}{2}g_{011} & 2g_{011}-G_{34} & G_{35} & g_{101}-\frac{1}{2}g_{011} & -2G_{46}-g_{020}-g_{002} & -\frac{1}{2}g_{011} \\
G_{16} & 2g_{002}-G_{35} & 0 & G_{46} & -\frac{1}{2}g_{011} & -g_{002}
\end{array}
\right]
$
}
\end{center}
Let $p_1=(1,1,0)$ and $p_2=(1,-1,0)$. Interlacers in ${\rm Int}(f,e)$ vanishing in $p_1$ and $p_2$ can be computed through the quantified linear matrix inequality
\begin{equation}
  \label{quantLMI}
\exists \, G_{ij} \,\,\,\,: \,\, g(p_1)=g(p_2)=0, \,\,\,\, G \succeq 0.
\end{equation}
Solving \eqref{quantLMI} symbolically using \cite{henrion2017spectra} yields the following
parametrization of an interlacer: $g = x^2-y^2+t\cdot z^2$, where $t$ is any of the two real
roots $t_1,t_2$ of $q(t)=49t^4-20t^3+22t^2+12t+1$ (the green curves in Figure \ref{expExact}).

Since the matrices $G$ corresponding to the two interlacers have rank $2$, the corresponding
Wronskian polynomials are sums of two squares. Choosing a rational $t_1 < r < t_2$ gives a
rational interlacer, for instance $g = x^2-y^2-\frac{1}{5}z^2$.

As in Example \ref{ex:rational}, our construction yields rational $4 \times 4$ determinantal
representations of $f$ times a rational linear polynomial, that can be built from the interlacer
$g = x^2-y^2-\frac{1}{5}z^2$:
$$
\frac{24}{125} f \cdot \left({2}x-{}y\right) =
\det
\begin{pmatrix}
  5x+10y & -x-2y & -4z & 2z \\
  -x-2y & x & 0 & 0 \\
  -4z & 0 & 4x+2y & -2x-4y \\
  2z & 0 & -2x-4y & 4x+2y
\end{pmatrix}.
$$
The matrix on the right hand side of the previous equality gives a spectrahedral representation of $C(f,e)$ (the green region in Figure \ref{expExact}).
\end{Example}

In the following, we show how our construction gives a \textit{sum-of-squares decomposition}, i.e. a representation $B(f,g)=S^t S$ for some (not necessarily square) matrix $S$ with polynomial entries, for any curve $f$ hyperbolic with respect to $(1,0,0)$ and any strict interlacer $g$.

 We have seen that there is a basis $a_1,\ldots,a_N$ of $\R[x,y,z]_{d-1}$ with $a_1=g$ and real symmetric matrices $A,B,C$ of size $N$ such that $A$ is positive definite and \begin{align}\label{eqn:bez}(xA+yB+zC)\cdot a = \delta_1\cdot f\end{align} where $a=(a_1,\ldots,a_N)^t$ and $\delta_1\in\R^N$ is the first unit vector. Let us write $a=a_{0}x^{d-1}+\ldots+a_{d-1}$ for some $a_i\in\R[y,z]_i^N$ and let $S$ be the matrix with columns $a_d,\ldots,a_0$. We claim that $\textnormal{B}(f,g)=S^T A S$, which implies that $\textnormal{B}(f,g)$ is in the interior of the sums-of-squares cone. Indeed, by \cite[\S 3]{DetBez}, we have that $\textnormal{B}(f,\tilde{g})=S^T A S$ for some $\tilde{g}\in\R[x,y,z]_{d-1}$. Furthermore, taking the derivative of (\ref{eqn:bez}) yields:
 $$A\cdot a+(xA+yB+zC)\cdot \textnormal{D}_e a = \delta_1\cdot \textnormal{D}_e f.$$ Now it follows by multiplying with $a^t$ from the left and another application of (\ref{eqn:bez}) that $$a^t\cdot A\cdot a+f\cdot \delta_1^t\cdot \textnormal{D}_e a=a^t\cdot A\cdot a+a^t\cdot(xA+yB+zC)\cdot \textnormal{D}_e a =a^t\cdot \delta_1\cdot \textnormal{D}_e f.$$Thus by Remark \ref{rem:wron} $$\textnormal{W}(f,\tilde{g})+f\cdot \textnormal{D}_e g=g\cdot \textnormal{D}_e f \Rightarrow \textnormal{W}(f,\tilde{g})=\textnormal{W}(f,{g})$$ which implies $g=\tilde{g}$ since $f$ is square-free.
 
\begin{Remark}
 It has been shown in \cite{DetBez} that sum-of-squares representations of a B\'ezout matrix of a hyperbolic polynomial $f$ as above give rise to a definite determinantal representation of some multiple of $f$. Now we have seen that for every strict interlacer of a hyperbolic curve there is a sum-of-squares decomposition of the corresponding B\'ezout matrix which even gives rise to a spectrahedral representation of the hyperbolicity cone.
\end{Remark}

\def\cprime{$'$}

\end{document}